\documentclass{amsart}\usepackage{amssymb,graphicx}

\theoremstyle{plain}
\newtheorem{thm}{Theorem}

\newcommand{\T}{\textstyle}
\newcommand{\iC}{{\rm C}}\newcommand{\iE}{{\rm E}}
\newcommand{\cM}{\mathcal{M}}
\newcommand{\C}{\mathbb{C}}\newcommand{\N}{\mathbb{N}}
\newcommand{\Q}{\mathbb{Q}}\newcommand{\Z}{\mathbb{Z}}

\newcommand{\M}{\begin{smallmatrix}}\newcommand{\E}{\end{smallmatrix}}
\newcommand{\VS}[1]{\T\sum\limits_{#1}}
\newcommand{\VT}[1]{\T\bigoplus\limits_{#1}}
\newcommand{\s}{\scalebox{0.8}}
\newcommand{\6}{\;\;\;\;\;\;}

\begin{document}
\title[the graded ring of modular forms of small level]{An explicit structure of the graded ring of modular forms of small level}
\author{SAITO Hayato}
\author{SUDA Tomohiko}
\address{Department of Mathematics, Hokkaido University,
Sapporo 060-0810, Japan}
\email[SUDA tomohiko]{s053019@math.sci.hokudai.ac.jp}
\maketitle

\section{Introduction}

For each integers $k\ge0$ and $N\ge1$, let $\cM_k(N)$ be the space of all modular forms of weight $k$ with respect to the congruence subgroup
\[\Gamma_0(N)=\{(\M a&b\\c&d\E)\in{\rm SL}_2(\Z) \,|\, c\in N\Z\}\]
of level $N$(cf \cite[Definition 1.2.3]{DS}), and
\[\cM(N)=\VT{k=0}^\infty\cM_k(N)\]
be the graded ring of modular forms for $\Gamma_0(N)$. When $N=1$, it is well-known that, as a $\C$-algebra, $\cM(1)$ is generated by Eisenstein series $\iE_4$ and $\iE_6$ of weight 4 and 6, and these two forms are algebraically independent:
\[\cM(1)\simeq\C[\iE_4,\iE_6].\]
For each $N$, we note that, $\cM(N)$ is generated by finitely many modular forms (cf \cite{B}), however, it is not necessarily isomorphic to the polynomial ring. For instance, when $N=3$, we prove
\[\cM(3)\simeq\C[\iC_3,\alpha_3,\beta_3]/(\alpha_3^2-\iC_3\beta_3)\]
for some $\iC_3\in\cM_2(3)$, $\alpha_3\in\cM_4(3)$ and $\beta_3\in\cM_6(3)$ (cf. Theorem 2). The aim of this paper is to give the ring structure of $\cM(N)$ explicitly for
\[N=2,3,4,5,6,7,8,9,10,12,16,18,25.\]

The method we use is summarized as follows: First, for each $N$, we take some suitable forms $f_1,\cdots,f_h$ from $\cM(N)$. Then, for each $k$, we see that a basis $\{b_1,\cdots,b_d\}$ $(d=\dim\cM_k(N))$ of $\cM_k(N)$ is obtained by $f_i$'s (cf. \S3), thus, the natural homomorphism from the polynomial ring $\C[f_1,\cdots,f_h]$ to $\cM(N)$ is surjective. Second, in \S4 and 5, we show some relations between $f_i$'s, i.e., give some elements of the kernel of the above-mentioned homomorphism. Third, using the result in \S6, we calculate the Hilbert functions, that is, generating series of the dimensions, and we obtain the isomorhism in \S7. In this context, we may regard $\cM(N)$ as a subring of $\C[[q]]$, where $q=e^{2\pi iz}$ ($z\in\C,\,{\rm Im}z>0$), via the Fourier expansion. Our basis $\{b_1,\cdots,b_d\}$ satisfy that $b_j\in q^{j-1}+\C[[q]]q^j$ for each $j$ (cf. \S3), hence we see
\[\cM_k(N)\cap\C[[q]]q^d=\{0\},\]
which is similar to a result of Sturm \cite{S}. This property will be used for the proof of relations between modular forms.

It should be emphasized that, when $N\in\{1,2,3,4,5,6,7,8,9,10,12,16,18,25\}$, $\cM(N)$ is generated by some Eisenstein series of weight 2 or 4 or 6, and $f_i$'s are obtained by such forms. We shall review the Eisenstein series and its Fourier expansion in \S2. Moreover, we show $f_i\in\Z[[q]]$ for each $i$, and
\[b_j\in q^{j-1}+\Z[[q]]q^j\]
for each $j$ (cf. \S8). When $N=1$, such an integral basis was also taken by Lang \cite[Ch.X, Theorem 4.4]{L}.
\smallskip

\section{Eisenstein series}

For each even integer $k>0$, let $B_k$ be the $k$-th Bernoulli number, $\sigma_k(n)=\VS{d|n}d^k$ and
\[\iE_k=1-\T\frac{2k}{B_k}\VS{n=1}^\infty\sigma_{k-1}(n)q^n\]
the Eisenstein series of weight $k$. It is well-known that if $k\ge4$, $\iE_k\in\cM_k(1)$. In particular, the following three forms will play important roles in the sequel:
\begin{align*}\iE_2&=1-24\VS{n}\sigma_1(n)q^n,\\
\iE_4&=1+240\VS{n}\sigma_3(n)q^n,\\
\iE_6&=1-504\VS{n}\sigma_5(n)q^n.\end{align*}

For each $h\in\N$, we note that $\cM(N)\subset\cM(Nh)$. For each $f\in\cM_k(N)$, we define $f^{(h)}\in\cM_k(Nh)$ to be
\[f^{(h)}(q)=f(q^h).\]
For each $N>1$, we put
\[\iC_N=\T\frac1{(N-1,24)}(N\iE_2^{(N)}-\iE_2),\]
then we see $\iC_N\in\cM_2(N)$ (cf. \cite[Exercises 1.2.8]{DS}), and
\[\iC_N=\T\frac{N-1}{(N-1,24)}+\frac{24}{(N-1,24)}\VS{n}\tau_N(n)q^n\]
where $\tau_N(n)=\VS{d|n,N\nmid d}d$.
In addition, we see
\begin{align*}\iC_N^{(h)}&=\T\frac1{(N-1,24)}(N\iE_2^{(Nh)}-\iE_2^{(h)})\\&=\T\frac1{(N-1,24)h}\big(-(h\!-\!1,24)\iC_h+(Nh\!-\!1,24)\iC_{Nh}\big).\end{align*}
For each prime number $p$, we put
\[\alpha_p=\T\frac1{240}(\iE_4-\iE_4^{(p)}),\]
then we see $\alpha_p\in\cM_4(p)\cap(1+\Z[[q]]q)$.

For a primitive Dirichlet character $\chi$ mod $N$, put $\sigma_\chi(n)=\VS{d|n}\chi(d)d$ and
\[\iE_\chi=\VS{n}\sigma_{\chi^2}(n)\overline\chi(n)q^n,\]
then we have $\iE_\chi\in\cM_2(N^2)$ (cf. \cite[\S4.5 and 4.6]{DS}). For $N=3,4,5$ we denote by $\rho_N$ the non-trivial real character mod $N$. Moreover let $\chi_5$ be the Dirichlet character mod 5 such that $\chi_5(2)=\sqrt{-1}$, and
\begin{align*}\iE_{r_5}&=\T\frac12(\iE_{\chi_5}+\iE_{\overline{\chi_5}}),\\
\iE_{i_5}&=\T\frac1{2i}(\iE_{\chi_5}-\iE_{\overline{\chi_5}}).\end{align*}
Then, since $(\chi_5)^2=\rho_5$ and $(\chi_5)^3=\overline{\chi_5}$, we see
\begin{align*}\iE_{r_5}&=\VS{n\equiv1\bmod5}\sigma_{\rho_5}(n)q^n-\VS{n\equiv4\bmod5}\sigma_{\rho_5}(n)q^n,\\
\iE_{i_5}&=-\VS{n\equiv2\bmod5}\sigma_{\rho_5}(n)q^n+\VS{n\equiv3\bmod5}\sigma_{\rho_5}(n)q^n.\end{align*}

\section{Basis of $\cM_k(N)$}

For each $k$ and $N$, we would take a basis $\{b_1,\cdots,b_d\}$ $(d=\dim\cM_k(N))$ of $\cM_k(N)$, such that
\[b_i\in q^{i-1}+\C[[q]]q^i\6(1\le i\le d).\]
First, we write down the following dimension formulas for even $k\ge0$:
\begin{align*}\dim\cM_k(1)&=\big[\T\frac k{12}\big]+1-\delta_{12\Z+2}(k),\\
\dim\cM_k(2)&=\big[\T\frac k4\big]+1,\\
\dim\cM_k(3)&=\big[\T\frac k3\big]+1,\\
\dim\cM_k(4)&=\T\frac k2+1,\\
\dim\cM_k(5)&=2\big[\T\frac k4\big]+1,\\
\dim\cM_k(6)&=k+1,\\
\dim\cM_k(7)&=2\big[\T\frac k3\big]+1,\\
\dim\cM_k(8)&=k+1,\\
\dim\cM_k(9)&=k+1,\\
\dim\cM_k(10)&=k+2\big[\T\frac k4\big]+1,\\
\dim\cM_k(12)&=2k+1,\\
\dim\cM_k(16)&=2k+1,\\
\dim\cM_k(18)&=3k+1,\\
\dim\cM_k(25)&=2k+2\big[\T\frac k4\big]+1,\end{align*}
where [ ] denotes the Gauss symbol and
\[\delta_X(k)=\begin{cases} 1&\text{if }k\in X,\\0&\text{if }k\not\in X\end{cases}\]
(cf \cite[Theorem 3.5.1]{DS}).

When $N=1$, let
\[\Delta=\T\frac1{12^3}(\iE_4^3-\iE_6^2)\in\cM_{12}(1)\]
be the Ramanujan $\Delta$-function, then for each $l>0$ we see $\cM_{12l}(1)\supset\VS{i=0}^l\C\iE_4^{3(l-i)}\Delta^i=\VT{i=0}^l\C\iE_4^{3(l-i)}\Delta^i$, and therefore, comparing the dimensions on both sides induces
\[\cM_{12l}(1)=\VT{i=0}^l\C\iE_4^{3(l-i)}\Delta^i.\]
Similarly, we have
\begin{align*}\cM_{12l+4}(1)&=\cM_{12l}(1)\iE_4,\\
\cM_{12l+6}(1)&=\cM_{12l}(1)\iE_6,\\
\cM_{12l+8}(1)&=\cM_{12l}(1)\iE_4^2,\\
\cM_{12l+10}(1)&=\cM_{12l}(1)\iE_4\iE_6,\\
\cM_{12l+14}(1)&=\cM_{12l}(1)\iE_4^2\iE_6.\end{align*}
For each $N>1$ we shall take, in the rest of this section, some modular forms and represent the basis $\{b_j\}$ by such forms.

\subsection{The case N=4,6,8,9,12,16,18}

In these cases, we have seen that
\[\dim\cM_k(N)=s\T\frac k2+1\]
for even $k\ge0$, where $s=\dim\cM_2(N)-1$. We take a $(s+1)$-tuple $(f_0,f_1,\cdots,f_s)$ satisfying
\[f_i\in\cM_2(N)\cap(q^i+\C[[q]]q^{i+1})\6(i=0,1,\cdots,s).\]
Then we have
\[\cM_{2l}(N)=\VT{i=0}^l\C f_0^{l-i}f_1^i\oplus\VT{i=1}^l\C f_1^{l-i}f_2^i\oplus\cdots\oplus\VT{i=1}^l\C f_{s-1}^{l-i}f_s^i.\]
Indeed, for a given $N$, such tuple can be taken as follows:

When $N=4$, $(f_0,f_1)=(\iC_4,\alpha_4)$, where
\[\alpha_4=\T\frac1{16}(\iC_2-\iC_4).\]

When $N=6$, $(f_0,f_1,f_2)=(\iC_3^{(2)},\alpha_6,\beta_6)$, where
\begin{align*}\alpha_6&=\T\frac1{12}(\iC_2-\iC_3),\\
\beta_6&=\T\frac1{12}(\iC_3^{(2)}-\iC_2^{(3)}).\end{align*}

When $N=8$, $(f_0,f_1,f_2)=(\iC_4^{(2)},\alpha_4,\alpha_4^{(2)})$.

When $N=9$, $(f_0,f_1,f_2)=(\iC_3,\iE_{\rho_3},\beta_9)$, where
\[\beta_9=\T\frac16(\frac19(\iC_3-\iC_9)-\iE_{\rho_3}).\]

When $N=12$, $(f_0,f_1,f_2,f_3,f_4)=(\iC_3^{(2)},\alpha_6,\beta_6,\alpha_4^{(3)},\beta_6^{(2)})$.

When $N=16$, $(f_0,f_1,f_2,f_3,f_4)=(\iC_4^{(2)},\alpha_4,\alpha_4^{(2)},\gamma_{16},\alpha_4^{(4)})$, where
\[\gamma_{16}=\T\frac18(\alpha_4-\iE_{\rho_4}).\]

When $N=18$, $(f_0,f_1,f_2,f_3,f_4,f_5,f_6)=\big(\iC_9^{(2)},\alpha_6,\beta_6,\alpha_6^{(3)},\beta_9^{(2)},\epsilon_{18},\beta_6^{(3)}\big)$, where
\[\epsilon_{18}=\T\frac12(\beta_9-\iE_{\rho_3}^{(2)}-3\beta_9^{(2)})+\beta_6^{(3)}.\]

\subsection{The case N=2,5,10,25}

In these cases, we have seen that
\[\dim\cM_k(N)=s\T\frac k2+t\big[\frac k4\big]+1,\]
where $s=\dim\cM_2(N)-1$ and $t=\dim\cM_4(N)-2s-1$. We take a $(s+t+1)$-tuple $(f_0,f_1,\cdots,f_s;g_1,g_2,\cdots,g_t)$ satisfying
\begin{align*}f_i&\in\cM_2(N)\cap(q^i+\C[[q]]q^{i+1})\6(i=0,1,\cdots,s),\\
g_i&\in\cM_4(N)\cap(q^{2s+i}+\C[[q]]q^{2s+i+1})\6(i=1,\cdots,t).\end{align*}
Then we have
\begin{align*}\cM_{4l}(N)&=\VT{i=0}^{2l}\C f_0^{2l-i}f_1^i\oplus\VT{i=1}^{2l}\C f_1^{2l-i}f_2^i\oplus\cdots\oplus\VT{i=1}^{2l}\C f_{s-1}^{2l-i}f_s^i\\[-2pt]
&\6\oplus\VT{i=0}^l\C f_s^{2(l-i)}g_1^i\oplus\VT{i=1}^l\C g_1^{l-i}g_2^i\oplus\cdots\oplus\VT{i=1}^l\C g_{t-1}^{l-i}g_t^i,\\[2pt]
\cM_{4l+2}(N)&=\cM_{4l}(N)f_0\oplus\C g_t^lf_1\oplus\cdots\oplus\C g_t^lf_s.\end{align*}
Indeed, such tuple can be taken as follows:

When $N=2$, $(f_0;g_1)=(\iC_2;\alpha_2)$.

When $N=5$, $(f_0;g_1,g_2)=(\iC_5;\alpha_5,\beta_5)$, where
\[\beta_5=\T\frac1{36}(-\iC_5^2+12\alpha_5+\iE_4^{(5)}).\]

When $N=10$, $(f_0,f_1,f_2;g_1,g_2)=(\iC_2,\alpha_{10},\beta_{10};\alpha_2^{(5)},\zeta_{10})$, where
\begin{align*}\alpha_{10}&=\T\frac18(\iC_2-4\iC_5+\iC_{10}),\\
\beta_{10}&=\T\frac16(\iC_5^{(2)}-\iC_2^{(5)}),\\
\zeta_{10}&=\T\frac14(\beta_{10}^2-\beta_5^{(2)}).\end{align*}

When $N=25$, $(f_0,f_1,f_2,f_3,f_4;g_1,g_2)=(\iC_5,\iE_{\rho_5},\iE_{i_5},\gamma_{25},\delta_{25};\iota_{25},\beta_5^{(5)})$, where
\begin{align*}\gamma_{25}&=\T\frac1{10}(-\iE_{\rho_5}+\iE_{r_5}-3\iE_{i_5}),\\
\delta_{25}&=\T\frac1{100}(\iC_5-\iC_{25}+5\iE_{\rho_5}-10\iE_{r_5}),\\
\iota_{25}&=\alpha_5^{(5)}+(-\iE_{r_5}+\gamma_{25}-2\delta_{25})\delta_{25}-\beta_5^{(5)}.\end{align*}

\subsection{The case N=3,7}

In these cases, we have seen that
\[\dim\cM_k(N)=s\big[\T\frac k3\big]+1,\]
where $s=\dim\cM_4(N)-1$. We take a $(2s+1)$-tuple $(f_0;g_1,\cdots,g_s;h_1,\cdots,h_s)$ satisfying
\begin{align*}f_0&\in\cM_2(N)\cap(1+\C[[q]]q),\\
g_i&\in\cM_4(N)\cap(q^i+\C[[q]]q^{i+1})\6(i=1,\cdots,s),\\
h_i&\in\cM_6(N)\cap(q^{s+i}+\C[[q]]q^{s+i+1})\6(i=1,\cdots,s).\end{align*}
Then we have
\begin{align*}\cM_{6l}(N)&=\VT{i=0}^l\C f_0^{3l-2i}g_1^i\oplus\VT{i=1}^l\C f_0^lg_1^{l-i}g_2^i\oplus\cdots\oplus\VT{i=1}^l\C f_0^lg_{s-1}^{l-i}g_s^i\\[-2pt]
&\6\oplus\VT{i=1}^l\C(f_0g_s)^{l-i}h_1^i\oplus\VT{i=1}^l\C h_1^{l-i}h_2^i\oplus\cdots\oplus\VT{i=1}^l\C h_{s-1}^{l-i}h_s^i,\\
\cM_{6l+2}(N)&=\cM_{6l}(N)f_0,\\
\cM_{6l+4}(N)&=\cM_{6l}(N)f_0^2\oplus\C h_s^lg_1\oplus\cdots\oplus\C h_s^lg_s.\end{align*}
Indeed, such tuple can be taken as follows:

When $N=3$, $(f_0;g_1;h_1)=(\iC_3;\alpha_3;\beta_3)$, where
\[\beta_3=\T\frac1{12}\big(\frac1{504}(\iE_6^{(3)}-\iE_6)-\iC_3\alpha_3\big).\]

When $N=7$, $(f_0;g_1,g_2;h_1,h_2)=(\iC_7;\alpha_7,\beta_7;\gamma_7,\delta_7)$, where
\begin{align*}\beta_7&=\T\frac1{32}(-\iC_7^2+8\alpha_7+\iE_4^{(7)}),\\
\gamma_7&=\T\frac1{360}\big(\frac{29}2(\iC_7\iE_4^{(7)}-\iE_6^{(7)})+\frac{17}{504}(\iE_6^{(7)}-\iE_6)-75\iC_7\alpha_7+240\iC_7\beta_7\big),\\
\delta_7&=\T\frac1{360}\big(\frac72(\iC_7\iE_4^{(7)}-\iE_6^{(7)})+\frac1{504}(\iE_6^{(7)}-\iE_6)-15\iC_7\alpha_7+120\iC_7\beta_7\big).\end{align*}

\section{Relations between modular forms for N dividing 12}

We define
\[O_3=\alpha_3^2-\iC_3\beta_3,\]
\[O_6=\alpha_6^2-\iC_3^{(2)}\beta_6,\]
\begin{align*} O_{12b}&=\gamma_{12}^2-\beta_6\beta_6^{(2)},\\
O_{12c}&=\iC_3^{(2)}\gamma_{12}-(\beta_6+2\gamma_{12}+4\beta_6^{(2)})\alpha_6,\\
O_{12d}&=\alpha_6\gamma_{12}-(\beta_6+2\gamma_{12}+4\beta_6^{(2)})\beta_6,\\
O_{12e}&=\alpha_6\beta_6^{(2)}-(\beta_6+2\gamma_{12}+4\beta_6^{(2)})\gamma_{12},\\
O_{12f}&=\iC_3^{(2)}\beta_6^{(2)}-\beta_6^2-4\alpha_6\beta_6^{(2)}-4\gamma_{12}^2-16\beta_6^{(2)2},\end{align*}
where
\[\gamma_{12}=\alpha_4^{(3)}-\beta_6^{(2)}.\]
In this section, we give some relations between modular forms, in particular, show that all forms defined above are 0. For $N=1$, we have seen that for each $k\ge8$, $\iE_k$ can be represented by $\iE_4$ and $\iE_6$. For example, we get $\iE_8,\iE_4^2\in\cM_8(1)\cap(1+\C[[q]]q)$, and thus $\iE_8-\iE_4^2\in\cM_8(1)\cap\C[[q]]q=\{0\}$, that is
\[\iE_8=\iE_4^2.\]

\subsection{The cases N=4,6,12}

First, we have $O_6\in\cM_4(6)\cap\C[[q]]q^5=\{0\}$ and $O_{12b}\in\cM_4(12)\cap\C[[q]]q^9=\{0\}$. We can show $O_{12c}=O_{12d}=O_{12e}=O_{12f}=0$ in a similar fashion, we also give algebraic proofs of them. We see
\begin{align*}\iC_2&=\iC_3^{(2)}+24\alpha_6+36\beta_6,\\
\iC_3&=\iC_3^{(2)}+12\alpha_6+36\beta_6,\\
\iC_4&=\iC_3^{(2)}+8\alpha_6+20\beta_6+16\gamma_{12}-16\beta_6^{(2)},\\
\iC_6&=5\iC_3^{(2)}+24\alpha_6+36\beta_6,\\
\iC_{12}&=11\iC_3^{(2)}+24\alpha_6-36\beta_6-144\gamma_{12}-144\beta_6^{(2)},\end{align*}
and
\begin{align*}0&=-O_6^{(2)}\\
&=-\alpha_6^{(2)2}+\iC_3^{(4)}\beta_6^{(2)}\\
&=-\big(\T\frac1{12}(\frac12(3\iC_4-\iC_2)-\iC_3^{(2)})\big)^2+\frac18(-3\iC_4+\iC_{12})\beta_6^{(2)}\\
&=\iC_3^{(2)}\beta_6^{(2)}-\beta_6^2-4\beta_6\gamma_{12}-8\beta_6\beta_6^{(2)}-4\gamma_{12}^2-16\gamma_{12}\beta_6^{(2)}-16\beta_6^{(2)2}\\
&=8O_{12b}+4O_{12e}+O_{12f}.\end{align*}
We also see
\begin{align*}\alpha_6\gamma_{12}O_{12f}&=(\iC_3^{(2)}\gamma_{12}+\alpha_6(\beta_6+4\beta_6^{(2)})-2\alpha_6\gamma_{12})O_{12e}-(\beta_6 + 2\gamma_{12} + 4\beta_6^{(2)})\beta_6^{(2)}O_6\\
&\6+(\iC_3^{(2)}\beta_6+2\iC_3^{(2)}\gamma_{12}+4\iC_3^{(2)}\beta_6^{(2)}-8\alpha_6\gamma_{12})O_{12b},\end{align*}
and thus
\[0=\big(\iC_3^{(2)}\gamma_{12}+\alpha_6(\beta_6+4\beta_6^{(2)})+2\alpha_6\gamma_{12}\big)O_{12e}.\]
We note that $\C[[q]]$ is a integral domain, and hence $0=O_{12e}=O_{12f}$. Moreover we have $O_{12d}=\frac{\beta_6}{\gamma_{12}}O_{12e}=0$ and $O_{12c}=\frac{\iC_3^{(2)}}{\alpha_6}O_{12d}=0$.\\

In $\cM(4)$ we get
\[\iE_4=\iC_2^2+192\iC_4\alpha_4\]
since $\iE_4-(\iC_2^2+192\iC_4\alpha_4)\in\cM_4(4)\cap\C[[q]]q^3=\{0\}$. In addition, considering in $\cM(12)$, we get a relation in $\cM(6)$:
\begin{align*}\iE_4&=\iC_2^2+12\iC_2\iC_4-12\iC_4^2\\
&=\s{$\iC_3^{(2)2}+240\iC_3^{(2)}\alpha_6+264\iC_3^{(2)}\beta_6-192\iC_3^{(2)}\gamma_{12}+192\iC_3^{(2)}\beta_6^{(2)}$}\\[-2pt]
&\6\s{$+2112\alpha_6^2+7104\alpha_6\beta_6+1536\alpha_6\gamma_{12}-1536\alpha_6\beta_6^{(2)}+5136\beta_6^2$}\\[-2pt]
&\6\s{$-768\beta_6\gamma_{12}+768\beta_6\beta_6^{(2)}-3072\gamma_{12}^2+6144\gamma_{12}\beta_6^{(2)}-3072\beta_6^{(2)2}$}\\
&=\s{$\iC_3^{(2)2}+240\iC_3^{(2)}\alpha_6+792\iC_3^{(2)}\beta_6+1584\alpha_6^2+6912\alpha_6\beta_6+6480\beta_6^2$}\\[-2pt]
&\6\s{$+48(11O_6-112O_{12b}-4O_{12c}+24O_{12d}-32O_{12e}+4O_{12f})$}\\
&=\iC_2^2+\iC_2\iC_3-\iC_3^2+5\iC_3\iC_6-\iC_6^2.\end{align*}

\subsection{The case N=2}

We see
\begin{align*}\iE_4^{(2)}&=(\iC_2^2+\iC_2\iC_3-\iC_3^2+5\iC_3\iC_6-\iC_6^2)^{(2)}\\
&=\s{$\iC_3^{(2)2}+216\iC_3^{(2)}\beta_6+432\iC_3^{(2)}\gamma_{12}+288\iC_3^{(2)}\beta_6^{(2)}-1152\beta_6^2-4608\beta_6\gamma_{12}$}\\[-2pt]
&\6\s{$-5760\beta_6\beta_6^{(2)}-4608\gamma_{12}^2-11520\gamma_{12}\beta_6^{(2)}-4608\beta_6^{(2)2}$}\\
&=\s{$\iC_3^{(2)2} -108\iC_3^{(2)}\beta_6+ 324\alpha_6^2 + 432\alpha_6\beta_6$}\\[-2pt]
&\6\s{$+36(-9O_6+64O_{12b}+12O_{12c}+24O_{12d}+80O_{12e}+8O_{12f})$}\\
&=\T\frac14(5\iC_2^2-(\iC_2^2+\iC_2\iC_3-\iC_3^2+5\iC_3\iC_6-\iC_6^2))\\
&=\T\frac14(5\iC_2^2-\iE_4).\end{align*}
We note $\iE_4^{(2)}=\iC_2^2-48\iC_4\alpha_4$ and $\alpha_2=\frac1{240}(\iE_4-\iE_4^{(2)})=\iC_4\alpha_4$.

We get
\[\iE_6=\iC_2(4\iC_2^2-3\iE_4)\]
since $\iE_6-\iC_2(4\iC_2^2-3\iE_4)\in\cM_6(2)\cap\C[[q]]q^2=\{0\}$. We also see
\begin{align*}\iE_6^{(2)}&=\iC_2^{(2)}(4\iC_2^{(2)2}-3(\iC_2^2-48\iC_4\alpha_4))\\
&=(\iC_4-8\alpha_4)(\iC_4+16\alpha_4)(\iC_4-32\alpha_4)\\
&=(\iC_4+16\alpha_4)(\iC_2^2-72\alpha_2)\\
&=\T\frac18\iC_2(11\iC_2^2-3\iE_4),\end{align*}
since $\iC_2=\iC_4+16\alpha_4$ and $\iC_2^{(2)}=\iC_4-8\alpha_4$.

\subsection{The case N=3}

We see
\begin{align*}\iE_4^{(3)}&=(\iC_2^2+12\iC_2\iC_4-12\iC_4^2)^{(3)}\\
&=\s{$\iC_3^{(2)2}-24\iC_3^{(2)}\beta_6+192\iC_3^{(2)}\gamma_{12}+192\iC_3^{(2)}\beta_6^{(2)}+144\beta_6^2-2304\beta_6\gamma_{12}$}\\[-2pt]
&\6\s{$-2304\beta_6\beta_6^{(2)}-3072\gamma_{12}^2-6144\gamma_{12}\beta_6^{(2)}-3072\beta_6^{(2)2}$}\\
&=\s{$\iC_3^{(2)2}-8\iC_3^{(2)}\beta_6 - 16\alpha_6^2 + 192\alpha_6\beta_6 + 720\beta_6^2$}\\[-2pt]
&\6\s{$+16(O_6+48O_{12b}+12O_{12c}+24O_{12d}+96O_{12e}+12O_{12f})$},\\
&=\T\frac19(10\iC_3^2-(\iC_2^2+\iC_2\iC_3-\iC_3^2+5\iC_3\iC_6-\iC_6^2))\\
&=\T\frac19(10\iC_3^2-\iE_4),\end{align*}
and
\begin{align*}\iE_6^{(3)}&=(\iC_2(4\iC_2^2-3\iE_4))^{(3)}\\
&=4\iC_2^{(3)3}-3\iC_2^{(3)}(\iC_3^{(2)2}-8\iC_2^{(3)}\beta_6 - 16\alpha_6^2 + 192\alpha_6\beta_6 + 720\beta_6^2)\\
&=\s{$\iC_3^{(2)3} - 84\iC_3^{(2)2}\beta_6 + 48\iC_3^{(2)}\alpha_6^2 - 576\iC_3^{(2)}\alpha_6\beta_6 - 720\iC_3^{(2)}\beta_6^2 - 576\alpha_6^2\beta_6$}\\[-2pt]
&\6\s{$ + 6912\alpha_6\beta_6^2 + 19008\beta_6^3$}\\
&=\s{$\iC_3^{(2)3} + \frac43\iC_3^{(2)2}\beta_6 - \frac{112}3\iC_3^{(2)}\alpha_6^2 - 64\iC_3^{(2)}\alpha_6\beta_6 - 720\iC_3^{(2)}\beta_6^2 - 512\alpha_6^3 - 576\alpha_6^2\beta_6$}\\[-2pt]
&\6\s{$ + 6912\alpha_6\beta_6^2 + 19008\beta_6^3+\T\frac{256}3(\iC_3^{(2)} + 6\alpha_6)O_6$}\\
&=\T\frac1{27}(35\iC_3^3-7\iC_3\iE_4-\iC_2(4\iC_2^2-3\iE_4))\\
&=\T\frac1{27}(35\iC_3^3-7\iC_3\iE_4-\iE_6).\end{align*}
Finally, since $\beta_6=\alpha_6^2/\iC_3^{(2)}$, we have\\

\[\iC_3=(\iC_3^{(2)}+6\alpha_6)^2/\iC_3^{(2)},\]
\begin{align*}\alpha_3&=\T\frac1{240}\big(\iE_4-\frac19(10\iC_3^2-\iE_4)\big)\\
&=\T\frac1{6^3}(\iE_4-\iC_3^2)\\
&=\iC_3^{(2)}\alpha_6+10\alpha_6^2+28\alpha_6\beta_6+24\beta_6^2\\
&=(\iC_3^{(2)} + 2\alpha_6)^2(\iC_3^{(2)} + 6\alpha_6)\alpha_6/\iC_3^{(2)2},\end{align*}
\begin{align*}\beta_3&=\T\frac1{12}\big(\frac1{504}(\frac1{27}(35\iC_3^3-7\iC_3\iE_4-\iE_6)-\iE_6)-\iC_3\frac1{6^3}(\iE_4-\iC_3^2)\big)\\
&=\T\frac1{108^2}(7\iC_3^3-5\iC_3\iE_4-2\iE_6)\\
&=\iC_3^{(2)}\alpha_6^2+8\alpha_6^3+24\alpha_6^2\beta_6+32\alpha_6\beta_6^2+16\beta_6^3\\
&=(\iC_3^{(2)} + 2\alpha_6)^4\alpha_6^2/\iC_3^{(2)3},\end{align*}
thus $\alpha_3^2=\iC_3\beta_3$ i.e. $O_3=0$.

\section{Relations between modular forms for N not-dividing 12}

Put
\begin{align*}\alpha_9&=\iE_{\rho_3}+9\beta_9,\\
u_{10}&=\T\frac13(-2\iC_2+5\iC_5),\\
\epsilon_{10}&=\alpha_2^{(5)}-5\zeta_{10},\\
u_{18}&=\iC_9^{(2)}-3\beta_6,\\
\alpha_{18}&=\alpha_6+3\beta_6,\\
\gamma_{18}&=\alpha_{18}^{(3)},\\
\delta_{18}&=\beta_9^{(2)}-\epsilon_{18}+2\beta_6^{(3)},\\
u_{25}&=\iC_{25}-5\iE_{r_5}-25\delta_{25},\\
\alpha_{25}&=\iE_{\rho_5}+5\gamma_{25}.\end{align*}
We define
\[O_5=\alpha_5^2-(\iC_5^2+4\alpha_5-8\beta_5)\beta_5,\]
\begin{align*} O_{7a}&=\beta_7^2-\iC_7\delta_7,\\
O_{7b}&=\iC_7\gamma_7-\alpha_7\beta_7,\\
O_{7c}&=\beta_7\gamma_7-\alpha_7\delta_7,\\
O_{7d}&=\alpha_7^2-(\iC_7^2+7\alpha_7-19\beta_7)\beta_7,\\
O_{7e}&=\alpha_7\gamma_7-(\iC_7\beta_7+7\gamma_7-19\delta_7)\beta_7,\\
O_{7f}&=\gamma_7^2-(\iC_7\beta_7+7\gamma_7-19\delta_7)\delta_7,\end{align*}
\[O_8=\alpha_4^2-\iC_4^{(2)}\alpha_4^{(2)},\]
\[O_9=\alpha_9^2-\iC_3\beta_9,\]
\begin{align*} O_{10a}&=\alpha_{10}^2-u_{10}\beta_{10},\\
O_{10b}&=\alpha_{10}\epsilon_{10}-u_{10}\zeta_{10},\\
O_{10c}&=\beta_{10}\epsilon_{10}-\alpha_{10}\zeta_{10},\\
O_{10d}&=\alpha_{10}\beta_{10}^2-(u_{10}+8\alpha_{10}+20\beta_{10})\epsilon_{10},\\
O_{10e}&=\beta_{10}^3-(\alpha_{10}\epsilon_{10}+8\beta_{10}\epsilon_{10}+20\beta_{10}\zeta_{10}),\\
O_{10f}&=\beta_{10}^2\zeta_{10}-(\epsilon_{10}^2+8\epsilon_{10}\zeta_{10}+20\zeta_{10}^2),\end{align*}
\begin{align*} O_{16b}&=\gamma_{16}^2-\alpha_4^{(2)}\alpha_4^{(4)},\\
O_{16c}&=\iC_4^{(2)}\gamma_{16}-(\alpha_4^{(2)}+4\alpha_4^{(4)})\alpha_4,\\
O_{16d}&=\alpha_4\gamma_{16}-(\alpha_4^{(2)}+4\alpha_4^{(4)})\alpha_4^{(2)},\\
O_{16e}&=\alpha_4\alpha_4^{(4)}-(\alpha_4^{(2)}+4\alpha_4^{(4)})\gamma_{16},\\
O_{16f}&=\iC_4^{(2)}\alpha_4^{(4)}-(\alpha_4^{(2)}+4\alpha_4^{(4)})^2,\end{align*}
\begin{align*} O_{18a}&=u_{18}\gamma_{18}-\alpha_{18}\beta_6,\\
O_{18b}&=u_{18}\epsilon_{18}-\alpha_{18}\delta_{18},\\
O_{18c}&=\gamma_{18}^2-u_{18}\beta_6^{(3)},\\
O_{18d}&=\beta_6\epsilon_{18}-\gamma_{18}\delta_{18},\\
O_{18e}&=\delta_{18}^2-\beta_6\beta_6^{(3)},\\
O_{18f}&=\delta_{18}\epsilon_{18}-\gamma_{18}\beta_6^{(3)},\\
O_{18A}&=(u_{18}+3\beta_6)\delta_{18}-\beta_6(\beta_6+3\gamma_{18}),\\
O_{18B}&=(\alpha_{18}+3\gamma_{18})\delta_{18}-\gamma_{18}(\beta_6+3\gamma_{18}),\\
O_{18C}&=(\alpha_{18}+3\gamma_{18})\epsilon_{18}-\beta_6^{(3)}(u_{18}+3\alpha_{18}),\\
O_{18D}&=\alpha_{18}\gamma_{18}-u_{18}(\delta_{18}+3\epsilon_{18}-3\beta_6^{(3)}),\\
O_{18E}&=\gamma_{18}^2-\beta_6(\delta_{18}+3\epsilon_{18}-3\beta_6^{(3)}),\\
O_{18F}&=\alpha_{18}\beta_6^{(3)}-\gamma_{18}(\delta_{18}+3\epsilon_{18}-3\beta_6^{(3)}),\\
O_{18G}&=\gamma_{18}\epsilon_{18}-\delta_{18}(\delta_{18}+3\epsilon_{18}-3\beta_6^{(3)}),\\
O_{18H}&=\epsilon_{18}^2-\beta_6^{(3)}(\delta_{18}+3\epsilon_{18}-3\beta_6^{(3)}),\\
O_{18I}&=\alpha_{18}^2-u_{18}\beta_6-3\alpha_{18}\gamma_{18}-6\alpha_{18}\beta_6+9\gamma_{18}^2,\end{align*}
\begin{align*} O_{25a}&=\alpha_{25}^2-u_{25}\iE_{i_5},\\
O_{25b}&=\alpha_{25}\iE_{i_5}-u_{25}\gamma_{25},\\
O_{25c}&=\iE_{i_5}^2-\alpha_{25}\gamma_{25},\\
O_{25d}&=\iE_{i_5}^2-u_{25}\delta_{25},\\
O_{25e}&=\iE_{i_5}\gamma_{25}-\alpha_{25}\delta_{25},\\
O_{25f}&=\gamma_{25}^2-\iE_{i_5}\delta_{25},\end{align*}
\begin{align*}
O_{25A}&=u_{25}\iota_{25}-\alpha_{25}(\delta_{25}^2-5\beta_5^{(5)}),\\
O_{25B}&=\alpha_{25}\iota_{25}-\iE_{i_5}(\delta_{25}^2-5\beta_5^{(5)}),\\
O_{25C}&=\iE_{i_5}\iota_{25}-\gamma_{25}(\delta_{25}^2-5\beta_5^{(5)}),\\
O_{25D}&=\gamma_{25}\iota_{25}-\delta_{25}(\delta_{25}^2-5\beta_5^{(5)}),\\
O_{25E}&=u_{25}\beta_5^{(5)}-\alpha_{25}(\iota_{25}-2\beta_5^{(5)}),\\
O_{25F}&=\alpha_{25}\beta_5^{(5)}-\iE_{i_5}(\iota_{25}-2\beta_5^{(5)}),\\
O_{25G}&=\iE_{i_5}\beta_5^{(5)}-\gamma_{25}(\iota_{25}-2\beta_5^{(5)}),\\
O_{25H}&=\gamma_{25}\beta_5^{(5)}-\delta_{25}(\iota_{25}-2\beta_5^{(5)}),\\
O_{25I}&=(\delta_{25}^2-5\beta_5^{(5)})\beta_5^{(5)}-\iota_{25}(\iota_{25}-2\beta_5^{(5)}).\end{align*}

\subsection{The case N=8}

We have $O_8=0$, since $\iC_4=\iC_4^{(2)}+8\alpha_4+16\alpha_4^{(2)}$ and
\begin{align*}0&=\iE_4^{(2)(2)}-\iE_4^{(2)(2)}\\
&=(\iC_2^{(2)2}-48\iC_4^{(2)}\alpha_4^{(2)})-\T\frac14(5\iC_2^{(2)2}-(\iC_2^2-48\iC_4\alpha_4))\\
&=48O_8.\end{align*}

\subsection{The case N=16}

We have $O_{16b}\in\cM_4(16)\cap\C[[q]]q^9=\{0\}$. We see
\begin{align*}\iC_4\alpha_4^{(2)}&=(\iC_4^{(2)}+4\alpha_4)^2/\iC_4^{(2)}\cdot\alpha_4^2/\iC_4^{(2)}\\
&=\big((\iC_4^{(2)}+4\alpha_4)\alpha_4/\iC_4^{(2)}\big)^2\\
&=(\alpha_4+4\alpha_4^{(2)})^2\end{align*}
and thus $O_{16f}=\big(\iC_4\alpha_4^{(2)}-(\alpha_4+4\alpha_4^{(2)})^2\big)^{(2)}=0$. Moreover, we have $O_{16c}=\frac{\alpha_4^2}{\iC_4^{(2)}\gamma_{16}+(\alpha_4^{(2)}+4\alpha_4^{(4)})\alpha_4}O_{16f}=0$, $O_{16d}=\frac{\alpha_4}{\iC_4^{(2)}}O_{16c}=0$, and $O_{16e}=\frac{\alpha_4^{(4)}}{\gamma_{16}}O_{16d}=0$.

\subsection{The case N=9}

We get
\[\iE_4=\iC_3^2+6^3\iC_9\alpha_9\]
since $\iE_4-(\iC_3^2+6^3\iC_9\alpha_9)\in\cM_4(9)\cap\C[[q]]q^5=\{0\}$. We note $\iE_4^{(3)}=\iC_3^2-24\iC_9\alpha_9$ and thus $\alpha_3=\iC_9\alpha_9$. In $\C[\iC_3,\iC_9,\alpha_9,\iE_6]$, we get
\begin{align*} O_3&=\T\frac1{6^3}(27\iC_3^4-18\iC_3^2\iE_4-\iE_4^2-8\iC_3\iE_6)\\
&=\T\frac1{3^3}(\iC_3^4 - 540\iC_3^2\iC_9\alpha_9 - 5832\iC_9^2\alpha_9^2 - \iC_3\iE_6),\\
O_3^{(3)}&=\s{$\frac1{3^7}(\iC_3^4 - 4\iC_3^3\iC_9 - 540\iC_3^2\iC_9\alpha_9 - 864\iC_3\iC_9^3 + 2160\iC_3\iC_9^2\alpha_9 + 864\iC_9^4$}\\[-2pt]
&\6\s{$ + 7776\iC_9^3\alpha_9 - 5832\iC_9^2\alpha_9^2 - \iC_3\iE_6 + 4\iC_9\iE_6)$},\end{align*}
thus
\begin{align*}0&=\T\frac{27}4(O_3-3^4O_3^{(3)})/\iC_9\\
&=\iC_3^3 + 216\iC_3\iC_9^2 - 540\iC_3\iC_9\alpha_9 - 216\iC_9^3 - 1944\iC_9^2\alpha_9-\iE_6.\end{align*}
Hence, we see
\begin{align*}\beta_3&=\T\frac1{108^2}(7\iC_3^3-5\iC_3\iE_4-2\iE_6)\\
&=\T\frac1{27}\iC_9^2(-\iC_3 + \iC_9 + 9\alpha_9)\\
&=\iC_9^2\beta_9\end{align*}
and
\[\iC_3\beta_9=\iC_3\beta_3/\iC_9^2=\alpha_3^2/\iC_9^2=(\iC_9\alpha_9)^2/\iC_9^2=\alpha_9,\]
i.e. $O_9=0$.

\subsection{The case N=18}

We put
\begin{align*} O_{18B'}&=(u_{18}+3\beta_6)\epsilon_{18}-\gamma_{18}(\beta_6+3\gamma_{18}),\\
O_{18X}&=\alpha_{18}\gamma_{18}-\beta_6^2-3\gamma_{18}^2-6\beta_6\gamma_{18}+9\beta_6\beta_6^{(3)},\end{align*}
then we see
\begin{align*}O_{18B'}&=O_{18B}+O_{18b}+3O_{18d},\\
O_{18X}&=3O_{18c}+O_{18A}+3O_{18B'}+O_{18D}+3O_{18E}.\end{align*}
First have
\[O_{18c}=O_6^{(3)}=0.\]
We see
\begin{align*}0&=O_6\\
&=-u_{18}\beta_6+\alpha_{18}^2-6\alpha_{18}\beta_6-3\beta_6^2-18\beta_6\gamma_{18}+27\beta_6\beta_6^{(3)}\\
&=O_{18I}+3O_{18X},\end{align*}
\begin{align*}0&=\iE_4^{(3)}-\iE_4^{(3)}\\
&=\s{$(\iC_2^2+\iC_2\iC_3-\iC_3^2+5\iC_3\iC_6-\iC_6^2)^{(3)}-(\iC_3^2-24\iC_9\alpha_9)$}\\
&=\s{$16(-u_{18}\beta_6+12u_{18}\gamma_{18}+9u_{18}\beta_6^{(3)}+\alpha_{18}^2-18\alpha_{18}\beta_6-3\beta_6^2-18\beta_6\gamma_{18}+27\beta_6\beta_6^{(3)}-9\gamma_{18}^2)$}\\
&=16(12O_{18a}-9O_{18c}+O_{18I}+3O_{18X}),\end{align*}
thus $O_{18a}=0$. We see $O_{18X}=\frac{\gamma_{18}}{\alpha_{18}}O_{18I}$ thus $O_I=O_X=0$. Next, we also see $O_{18C}=\frac{\alpha_{18}}{u_{18}}O_{18B'}$, $O_{18B}=\frac{\alpha_{18}}{u_{18}}O_{18A}$ and
\begin{align*}0&=\iE_4-\iE_4\\
&=\s{$(\iC_2^2+\iC_2\iC_3-\iC_3^2+5\iC_3\iC_6-\iC_6^2)-(\iC_3^2+6^3\iC_9\alpha_9)$}\\
&=\s{$72(-11u_{18}\beta_6-21u_{18}\gamma_{18}+27u_{18}\delta_{18}+9u_{18}\epsilon_{18}+11\alpha_{18}^2-45\alpha_{18}\beta_6-36\alpha_{18}\gamma_{18}$}\\[-2pt]
&\6\s{$+81\alpha_{18}\delta_{18}+27\alpha_{18}\epsilon_{18}-81\alpha_{18}\beta_6^{(3)}-24\beta_6^2-153\beta_6\gamma_{18}+81\beta_6\delta_{18}+27\beta_6\epsilon_{18}$}\\[-2pt]
&\6\s{$-27\beta_6\beta_6^{(3)}-189\gamma_{18}^2+243\gamma_{18}\delta_{18}+81\gamma_{18}\epsilon_{18})$}\\
&=\s{$72(-21O_{18a}-36O_{18c}+24O_{18A}+81O_{18B}+27O_{18C}-3O_{18D}-9O_{18E}+11O_{18I})$}\\
&=72\cdot3(8O_{18A}+27O_{18B}+9O_{18C}-O_{18D}-3O_{18E})\\
&=72\cdot3\cdot3\cdot(3\T\frac{\alpha_{18}}{u_{18}}+1)(3O_{18A}+O_{18B'}),\end{align*}
\begin{align*}0&=\iE_4^{(2)}-\iE_4^{(2)}\\
&=\s{$\frac14(5\iC_2^2-(\iC_2^2+\iC_2\iC_3-\iC_3^2+5\iC_3\iC_6-\iC_6^2))-(\iC_3^2+6^3\iC_9\alpha_9)^{(2)}$}\\
&=\s{$108(-3u_{18}\beta_6-4u_{18}\gamma_{18}-6u_{18}\delta_{18}-6u_{18}\epsilon_{18}+18u_{18}\beta_6^{(3)}+3\alpha_{18}^2-14\alpha_{18}\beta_6$}\\[-2pt]
&\6\s{$-3\beta_6^2-30\beta_6\gamma_{18}-18\beta_6\delta_{18}-18\beta_6\epsilon_{18}+81\beta_6\beta_6^{(3)})$}\\
&=\s{$18(-4O_{18a}+9O_{18c}+3O_{18A}+21O_{18B'}+9O_{18D}+27O_{18E}+3O_{18I})$}\\
&=18\cdot3(O_{18A}+7O_{18B'}+3O_{18D}+9O_{18E})\\
&=-18\cdot3\cdot2(O_{18A}+O_{18B'}),\end{align*}
thus $O_{18A}=O_{18B'}=0$ and $O_{18B}=O_{18C}=0$. We see $O_{18d}=\frac{\gamma_{18}}{\alpha_{18}}O_{18b}$ and
\[O_{18b}+3O_{18d}=O_{18B'}-O_{18B}=0,\]
thus $O_{18b}=O_{18d}=0$. Moreover, we see $O_{18E}=\frac{\beta_6}{u_{18}}O_{18D}$ and
\[O_{18D}+3O_{18E}=O_{18X}=0,\]
thus $O_{18D}=O_{18E}=0$, $O_{18F}=\frac{\gamma_{18}}{u_{18}}O_{18D}=0$, $O_{18G}=\frac{\epsilon_{18}}{\gamma_{18}}O_{18E}=0$.
Finally we see
\[O_{18H}=\T\frac{\beta_6^{(3)}}{\delta_{18}}O_{18G}+\frac{\epsilon_{18}}{\delta_{18}}O_{18f}=\frac{\epsilon_{18}^2}{\delta_{18}^2}O_{18e},\]
\begin{align*}0&=O_9^{(2)}\\
&=\s{$-u_{18}\delta_{18}-u_{18}\epsilon_{18}+2u_{18}\beta_6^{(3)}+\beta_6^2+4\beta_6\gamma_{18}-6\beta_6\delta_{18}-6\beta_6\epsilon_{18}+6\beta_6\beta_6^{(3)}+4\gamma_{18}^2$}\\[-2pt]
&\6\s{$-6\gamma_{18}\delta_{18}-6\gamma_{18}\epsilon_{18}+9\delta_{18}^2+18\delta_{18}\epsilon_{18}-27\delta_{18}\beta_6^{(3)}+9\epsilon_{18}^2-27\epsilon_{18}\beta_6^{(3)}+27\beta_6^{(3)2}$}\\
&=\s{$-O_{18b}-2O_{18c}+3O_{18d}+3O_{18e}-O_{18A}-O_{18B}+3O_{18E}-6O_{18G}+9O_{18H}$}\\
&=3(1+3\T\frac{\epsilon_{18}^2}{\delta_{18}^2})O_{18e},\end{align*}
thus $O_{18e}=O_{18f}=O_{18H}=0$.

\subsection{The case N=5,10}

First, in $\cM(5)$ we get
\[\iE_6=\T\frac18\iC_5(2000\iC_5^2-117\iE_4-1875\iE_4^{(5)}).\]

We have $O_{10a}\in\cM_4(10)\cap\C[[q]]q^7=\{0\}$. We see $O_{10c}=\frac{\alpha_{10}}{u_{10}}O_{10b}$,
\[O_{10e}=\T\frac{\alpha_{10}}{u_{10}}O_{10d}-20\frac{\beta_{10}}{u_{10}}O_{10b},\]
and
\begin{align*}0&=\iE_6-\iE_6\\
&=\iC_2(4\iC_2^2-3\iE_4)-\T\frac18\iC_5(2000\iC_5^2-117\iE_4-1875\iE_4^{(5)})\\
&=\s{$12\big(-1595u_{10}^2\beta_{10}+1595u_{10}\alpha_{10}^2-14406u_{10}\alpha_{10}\beta_{10}-50612u_{10}\beta_{10}^2+14406\alpha_{10}^3$}\\[-2pt]
&\6\s{$+50612\alpha_{10}^2\beta_{10}-216\alpha_{10}\beta_{10}^2-29520\beta_{10}^3+216u_{10}\epsilon_{10}+12240u_{10}\zeta_{10}+19008\alpha_{10}\epsilon_{10}$}\\[-2pt]
&\6\s{$+177120\alpha_{10}\zeta_{10}+63360\beta_{10}\epsilon_{10}+590400\beta_{10}\zeta_{10}\big)$}\\
&=\s{$12\big((1595u_{10}+14406\alpha_{10}+50612\beta_{10})O_{10a}-72(170O_{10b}+2460O_{10c}+3O_{10d}+410O_{10e})\big)$}\\
&=12\cdot72\big((170+2460\T\frac{\alpha_{10}}{u_{10}}-8200\frac{\beta_{10}}{u_{10}})O_{10b}-(3+410\frac{\alpha_{10}}{u_{10}})O_{10d}\big),\end{align*}
\begin{align*}0&=\iE_6^{(2)}-\iE_6^{(2)}\\
&=\s{$\frac18\iC_2(11\iC_2^2-3\iE_4)-\frac18\iC_5^{(2)}\big(2000\iC_5^{(2)2}-\frac{117}4(5\iC_2^2-\iE_4)-\frac{1875}4(5\iC_2^2-\iE_4)^{(5)}\big)$}\\
&=\s{$3\big(455u_{10}^2\beta_{10}-455u_{10}\alpha_{10}^2+5058u_{10}\alpha_{10}\beta_{10}+16736u_{10}\beta_{10}^2-5058\alpha_{10}^3-16736\alpha_{10}^2\beta_{10}$}\\[-2pt]
&\6\s{$-1152\alpha_{10}\beta_{10}^2+1080\beta_{10}^3+1152u_{10}\epsilon_{10}-9720u_{10}\zeta_{10}+17856\alpha_{10}\epsilon_{10}-38160\alpha_{10}\zeta_{10}$}\\[-2pt]
&\6\s{$+52560\beta_{10}\epsilon_{10}-21600\beta_{10}\zeta_{10}\big)$}\\
&=\s{$3\big(-(455u_{10}+5058\alpha_{10}+16736\beta_{10})O_{10a}+72(135O_{10b}+530O_{10c}-16O_{10d}+15O_{10e})\big)$}\\
&=3\cdot72\big((135+530\T\frac{\alpha_{10}}{u_{10}}-300\frac{\beta_{10}}{u_{10}})O_{10b}-(16-15\frac{\alpha_{10}}{u_{10}})O_{10d}\big).\end{align*}
Thus we have $O_{10b}=O_{10c}=O_{10d}=O_{10e}=0$, and $O_{10f}=\frac{\zeta_{10}}{\beta_{10}}O_{10e}=0$.

We put
\[O_{10e'}=\beta_{10}^3-(u_{10}+8\alpha_{10}+20\beta_{10})\zeta_{10},\]
then we see $O_{10e'}=O_{10e}+O_{10b}+8O_{10c}=0$. We get
\begin{align*}\iE_6^{(5)}&=(\iC_2(4\iC_2^2-3\iE_4))^{(5)}\\
&=\s{$12\big(13u_{10}^2\beta_{10}-13u_{10}\alpha_{10}^2+162u_{10}\alpha_{10}\beta_{10}+556u_{10}\beta_{10}^2-162\alpha_{10}^3-556\alpha_{10}^2\beta_{10}+72\alpha_{10}\beta_{10}^2$}\\[-2pt]
&\6\s{$+432\beta_{10}^3-72u_{10}\epsilon_{10}-432u_{10}\zeta_{10}-576\alpha_{10}\epsilon_{10}-3744\alpha_{10}\zeta_{10}-1152\beta_{10}\epsilon_{10}-8640\beta_{10}\zeta_{10}\big)$}\\
&=\s{$u_{10}^3+18u_{10}^2\alpha_{10}-120u_{10}^2\beta_{10}+240u_{10}\alpha_{10}^2-1656u_{10}\alpha_{10}\beta_{10}-6528u_{10}\beta_{10}^2$}\\[-2pt]
&\6\s{$+2016\alpha_{10}^3+7008\alpha_{10}^2\beta_{10}-576\alpha_{10}\beta_{10}^2-5120\beta_{10}^3+288u_{10}\epsilon_{10}+2304u_{10}\zeta_{10}$}\\[-2pt]
&\6\s{$+3456\alpha_{10}\epsilon_{10}+27648\alpha_{10}\zeta_{10}+11520\beta_{10}\epsilon_{10}+92160\beta_{10}\zeta_{10}$}\\[-2pt]
&\6\s{$+12(-(162\alpha_{10}+556\beta_{10}+13u_{10})O_{10a}+72(4O_{10c}+O_{10d}+6O_{10e'}))$}\\
&=\T\frac1{40}\iC_5(-80\iC_5^2+3\iE_4+117\iE_4^{(5)}).\end{align*}
At last, we have
\[O_5=\T\frac1{12^4\cdot5}(3520\iC_5^4+\iE_4^2+625\iE_4^{(5)2}-160\iC_5^2\iE_4-4000\iC_5^2\iE_4^{(5)}+14\iE_4\iE_4^{(5)})\]
in $\C[\iC_5,\iE_4,\iE_4^{(5)}]$ and
\begin{align*} O_5&=\s{$\frac4{81}\big(45u_{10}^3\beta_{10}-45u_{10}^2\alpha_{10}^2+1098u_{10}^2\alpha_{10}\beta_{10}+4157u_{10}^2\beta_{10}^2-1098u_{10}\alpha_{10}^3$}\\[-2pt]
&\6\s{$+1694u_{10}\alpha_{10}^2\beta_{10}+42480u_{10}\alpha_{10}\beta_{10}^2+74160u_{10}\beta_{10}^3-5851\alpha_{10}^4-42480\alpha_{10}^3\beta_{10}$}\\[-2pt]
&\6\s{$-73188\alpha_{10}^2\beta_{10}^2+12960\alpha_{10}\beta_{10}^3+42768\beta_{10}^4-324u_{10}^2\zeta_{10}-648u_{10}\alpha_{10}\epsilon_{10}$}\\[-2pt]
&\6\s{$-10368u_{10}\alpha_{10}\zeta_{10}-2880u_{10}\beta_{10}\epsilon_{10}-43200u_{10}\beta_{10}\zeta_{10}-7488\alpha_{10}^2\epsilon_{10}-68256\alpha_{10}^2\zeta_{10}$}\\[-2pt]
&\6\s{$-54432\alpha_{10}\beta_{10}\epsilon_{10}-508032\alpha_{10}\beta_{10}\zeta_{10}-93312\beta_{10}^2\epsilon_{10}-860544\beta_{10}^2\zeta_{10}+5184\epsilon_{10}^2$}\\[-2pt]
&\6\s{$+41472\epsilon_{10}\zeta_{10}+103680\zeta_{10}^2\big)$}\\
&=\s{$\frac4{81}\big((-45u_{10}^2-1098u_{10}\alpha_{10}-4157u_{10}\beta_{10}-5851\alpha_{10}^2-42480\alpha_{10}\beta_{10}-74160\beta_{10}^2$}\\[-2pt]
&\6\s{$+2880\epsilon_{10}+432\zeta_{10})O_{10a}+324((u_{10}-8\alpha_{10})O_{10b}-36(3\alpha_{10}+8\beta_{10})O_{10c}+3\alpha_{10}O_{10d}$}\\
&\6\s{$+4(10\alpha_{10}+33\beta_{10})O_{10e'}-16O_{10f})\big)$}\\
&=0.\end{align*}

\subsection{The case N=25}

We have $O_{25d},O_{25e},O_{25f}\in\cM_4(25)\cap\C[[q]]q^{11}=\{0\}$, thus we have also
\[\T O_{25c}=\frac{\iE_{i_5}}{\gamma_{25}}O_{25e}=0,\;\;O_{25b}=\frac{\alpha_{25}}{\iE_{i_5}}O_{25d}=0,\;\;O_{25a}=\frac{\alpha_{25}}{\iE_{i_5}}O_{25b}=0.\]
We have $O_{25D},O_{25H}\in\cM_6(25)\cap\C[[q]]q^{15}=\{0\}$, thus
\[\T O_{25C}=\frac{\iE_{i_5}}{\gamma_{25}}O_{25D}=0,\;\;O_{25B}=\frac{\alpha_{25}}{\iE_{i_5}}O_{25C}=0,\;\;O_{25A}=\frac{u_{25}}{\alpha_{25}}O_{25B}=0,\]
\[\T O_{25G}=\frac{\iE_{i_5}}{\gamma_{25}}O_{25H}=0,\;\;O_{25F}=\frac{\alpha_{25}}{\iE_{i_5}}O_{25G}=0,\;\;O_{25E}=\frac{u_{25}}{\alpha_{25}}O_{25F}=0,\]
\[\T O_{25I}=\frac{\iota_{25}}{\delta_{25}}O_{25H}=0.\]
Moreover we get
\begin{align*}\T\frac15\iC_5(-80\iC_5^2+3\iE_4+117\iE_4^{(5)})&=8\iE_6^{(5)}\\
&=\iC_5^{(5)}(2000\iC_5^{(5)2}-117\iE_4^{(5)}-1875\iE_4^{(25)}),\end{align*}
thus

\begin{align*}\iE_4^{(25)}&=\T\frac1{1875\iC_5^{(5)}}\big(\iC_5^{(5)}(2000\iC_5^{(5)2}-117\iE_4^{(5)})-\frac15\iC_5(-80\iC_5^2+3\iE_4+117\iE_4^{(5)})\big)\\
&=\T\frac1{625\iC_5^{(5)}}\big(\s{$625u_{25}^3+7500u_{25}^2\alpha_{25}+18009u_{25}^2\iE_{i_5}+20410u_{25}^2\gamma_{25}+10805u_{25}^2\delta_{25}$}\\[-2pt]
&\6\s{$+28866u_{25}\alpha_{25}^2+141044u_{25}\alpha_{25}\iE_{i_5}+177680u_{25}\alpha_{25}\gamma_{25}+104440u_{25}\alpha_{25}\delta_{25}$}\\[-2pt]
&\6\s{$+165051u_{25}\iE_{i_5}^2+425488u_{25}\iE_{i_5}\gamma_{25}+360000u_{25}\iE_{i_5}\delta_{25}+212794u_{25}\gamma_{25}^2$}\\[-2pt]
&\6\s{$+427972u_{25}\gamma_{25}\delta_{25}+198961u_{25}\delta_{25}^2+32296\alpha_{25}^3+227714\alpha_{25}^2\iE_{i_5}+362740\alpha_{25}^2\gamma_{25}$}\\[-2pt]
&\6\s{$+223370\alpha_{25}^2\delta_{25}+419832\alpha_{25}\iE_{i_5}^2+1571832\alpha_{25}\iE_{i_5}\gamma_{25}+1762060\alpha_{25}\iE_{i_5}\delta_{25}$}\\[-2pt]
&\6\s{$+1132456\alpha_{25}\gamma_{25}^2+3316528\alpha_{25}\gamma_{25}\delta_{25}+1802164\alpha_{25}\delta_{25}^2-108621\iE_{i_5}^3-378738\iE_{i_5}^2\gamma_{25}$}\\[-2pt]
&\6\s{$+1991655\iE_{i_5}^2\delta_{25}-2675894\iE_{i_5}\gamma_{25}^2+5762228\iE_{i_5}\gamma_{25}\delta_{25}+7828569\iE_{i_5}\delta_{25}^2$}\\[-2pt]
&\6\s{$-5652860\gamma_{25}^3-6961110\gamma_{25}^2\delta_{25}+317070\gamma_{25}\delta_{25}^2+121205\delta_{25}^3-11532u_{25}\iota_{25}$}\\[-2pt]
&\6\s{$-8616u_{25}\beta_5^{(5)}-61968\alpha_{25}\iota_{25}-24384\alpha_{25}\beta_5^{(5)}-198828\iE_{i_5}\iota_{25}-17064\iE_{i_5}\beta_5^{(5)}$}\\[-2pt]
&\6\s{$-352920\gamma_{25}\iota_{25}+65040\gamma_{25}\beta_5^{(5)}-374460\delta_{25}\iota_{25}+158520\delta_{25}\beta_5^{(5)}$}\big)\\
&=\T\frac1{625\iC_5^{(5)}}\big(\s{$625u_{25}^3+7500u_{25}^2\alpha_{25}+18585u_{25}^2\iE_{i_5}+22650u_{25}^2\gamma_{25}+14325u_{25}^2\delta_{25}$}\\[-2pt]
&\6\s{$+28290u_{25}\alpha_{25}^2+137940u_{25}\alpha_{25}\iE_{i_5}+217200u_{25}\alpha_{25}\gamma_{25}+151800u_{25}\alpha_{25}\delta_{25}$}\\[-2pt]
&\6\s{$+110715u_{25}\iE_{i_5}^2+325200u_{25}\iE_{i_5}\gamma_{25}+489600u_{25}\iE_{i_5}\delta_{25}-63750u_{25}\gamma_{25}^2$}\\[-2pt]
&\6\s{$+226500u_{25}\gamma_{25}\delta_{25}+56625u_{25}\delta_{25}^2+33160\alpha_{25}^3+239010\alpha_{25}^2\iE_{i_5}+489300\alpha_{25}^2\gamma_{25}$}\\[-2pt]
&\6\s{$+369450\alpha_{25}^2\delta_{25}+346200\alpha_{25}\iE_{i_5}^2+1517400\alpha_{25}\iE_{i_5}\gamma_{25}+2239500\alpha_{25}\iE_{i_5}\delta_{25}$}\\[-2pt]
&\6\s{$+105000\alpha_{25}\gamma_{25}^2+1458000\alpha_{25}\gamma_{25}\delta_{25}+412500\alpha_{25}\delta_{25}^2-53325\iE_{i_5}^3+372750\iE_{i_5}^2\gamma_{25}$}\\[-2pt]
&\6\s{$+2886375\iE_{i_5}^2\delta_{25}-1569750\iE_{i_5}\gamma_{25}^2+3772500\iE_{i_5}\gamma_{25}\delta_{25}+1625625\iE_{i_5}\delta_{25}^2$}\\[-2pt]
&\6\s{$-2277500\gamma_{25}^3-813750\gamma_{25}^2\delta_{25}+168750\gamma_{25}\delta_{25}^2+3125\delta_{25}^3-7500u_{25}\iota_{25}$}\\[-2pt]
&\6\s{$+15000u_{25}\beta_5^{(5)}-30000\alpha_{25}\iota_{25}+60000\alpha_{25}\beta_5^{(5)}-67500\iE_{i_5}\iota_{25}+135000\iE_{i_5}\beta_5^{(5)}$}\\[-2pt]
&\6\s{$-75000\gamma_{25}\iota_{25}+150000\gamma_{25}\beta_5^{(5)}-37500\delta_{25}\iota_{25}+75000\delta_{25}\beta_5^{(5)}$}\\[-2pt]
&\6\s{$+32\big((18u_{25}-27\alpha_{25}-1588\iE_{i_5}-3955\gamma_{25}-4565\delta_{25})O_{25a}$}\\[-2pt]
&\6\6\s{$+(70u_{25}+1235\alpha_{25}+821\iE_{i_5})O_{25b}-1701\iE_{i_5}O_{25c}$}\\[-2pt]
&\6\6\s{$+(110u_{25}+1480\alpha_{25}-27\iE_{i_5}-6296\gamma_{25}-4448\delta_{25})O_{25d}$}\\[-2pt]
&\6\6\s{$-(17188\iE_{i_5}+58079\gamma_{25}+43301\delta_{25})O_{25e}$}\\[-2pt]
&\6\6\s{$+(8642u_{25}+32108\alpha_{25}+23512\iE_{i_5}-105480\gamma_{25}-192105\delta_{25})O_{25f}$}\\[-2pt]
&\6\6\s{$-126O_{25A}-1737O_{25B}-4635O_{25C}-3690O_{25D}-738O_{25E}-531O_{25F}$}\\[-2pt]
&\6\6\s{$+4995O_{25G}+10530O_{25H}\big)$}\big)\\
&=\T\frac1{625}(52(\iC_5-3\iC_{25})^2+432\iC_{25}^2-60\iE_{\rho_5}^2+600\iE_{r_5}^2+600\iE_{i_5}^2-\iE_4-14\iE_4^{(5)}).\end{align*}

\subsection{The case N=7}

We have $O_{7a},O_{7b},O_{7d}\in\cM_8(7)\cap\C[[q]]q^5=\{0\}$, thus we have also
\[O_{7c}=\T\frac{\beta_7}{\iC_7}O_{7b}=0,\;\;O_{7e}=\frac{\beta_7}{\iC_7}O_{7d}=0,\;\;O_{7f}=\frac{\beta_7}{\iC_7}O_{7e}=0.\]

\section{Decomposition of polynomial rings}

Suppose that $R$ is a ring. If $O\in R+RY+Y^2$, then by induction on $m$, we get $RY^m\subset(O)+R+RY$, thus
\[R[Y]=(O)\oplus R\oplus RY.\]
By virtue of this, for example, we see $\C[\iC_3^{(2)},\alpha_6,\beta_6]=(O_6)\oplus\C[\iC_3^{(2)},\beta_6]\oplus\C[\iC_3^{(2)},\beta_6]\alpha_6$. Next, if $O\in XZ+R[Z]$, then we see $R[Z]X\subset(R+R[Z]Z)X\subset RX+(O)+R[Z]$ and
\[R[X,Z]=(O)+R[X]+R[Z].\]
For example, we see $\C[\iC_3^{(2)},\alpha_6,\beta_6]=(O_6)+\C[\alpha_6][\iC_3^{(2)}]+\C[\alpha_6][\beta_6]$. If $O\in XZ+R[Z,Z']$ and $O'\in XZ'+R[Z,Z']$, then we get
\[R[X,Z,Z']=(O,O')+R[X]+R[Z,Z']\]
in a similar fashion. If $O\in XZ+R[Z,Z',Z'']$, $O'\in XZ'+R[Z,Z',Z'']$ and $O''\in XZ''+R[Z,Z',Z'']$, then we get
\[R[X,Z,Z',Z'']=(O,O',O'')+R[X]+R[Z,Z',Z''].\]

\subsection{The case N=12,16}

Put $R_{12}=\C[\iC_3^{(2)},\alpha_6,\beta_6,\gamma_{12},\beta_6^{(2)}]$ and
\[I_{12}=(O_6,O_{12b},O_{12c},O_{12d},O_{12e},O_{12f}).\]
We see
\[\C[\beta_6,\gamma_{12},\beta_6^{(2)}]=(O_{12b})+\C[\gamma_{12}][\beta_6]+\C[\gamma_{12}][\beta_6^{(2)}],\]
\[\C[\alpha_6,\beta_6,\gamma_{12},\beta_6^{(2)}]=(O_{12d},O_{12e})+\C[\beta_6][\alpha_6]+\C[\beta_6][\gamma_{12},\beta_6^{(2)}],\]
and
\begin{align*} R_{12}&=(O_{12c},O_{12f})+\C[\alpha_6,\beta_6][\iC_3^{(2)}]+\C[\alpha_6,\beta_6][\gamma_{12},\beta_6^{(2)}]\\
&=I_{12}+\C[\iC_3^{(2)},\alpha_6]+\C[\alpha_6,\beta_6]+\C[\beta_6,\gamma_{12}]+\C[\gamma_{12},\beta_6^{(2)}].\end{align*}

Put $R_{16}=\C[\iC_4^{(2)},\alpha_4,\alpha_4^{(2)},\gamma_{16},\alpha_4^{(4)}]$ and
\[I_{16}=(O_8,O_{16b},O_{16c},O_{16d},O_{16e},O_{16f}).\]
Similarly, we see
\[R_{16}=I_{16}+\C[\iC_4^{(2)},\alpha_4]+\C[\alpha_4,\alpha_4^{(2)}]+\C[\alpha_4^{(2)},\gamma_{16}]+\C[\gamma_{16},\alpha_4^{(4)}].\]

\subsection{The case N=18}

Put $R_{18}=\C[u_{18},\alpha_{18},\beta_6,\gamma_{18},\delta_{18},\epsilon_{18},\beta_6^{(3)}]$ and
\[I_{18}=(\s{$O_{18a},O_{18b},O_{18c},O_{18d},O_{18e},O_{18f},O_{18A},O_{18B},O_{18C},O_{18D},O_{18E},O_{18F},O_{18G},O_{18H},O_{18I}$}).\]
We see
\[\C[\delta_{18},\epsilon_{18},\beta_6^{(3)}]=(O_{18H})+\C[\epsilon_{18}][\delta_{18}]+\C[\epsilon_{18}][\beta_6^{(3)}],\]
\[\C[\gamma_{18},\delta_{18},\epsilon_{18},\beta_6^{(3)}]=(O_{18G},O_{18f})+\C[\delta_{18}][\gamma_{18}]+\C[\delta_{18}][\epsilon_{18},\beta_6^{(3)}],\]
and
\begin{align*}\C[\beta_6,\gamma_{18},\delta_{18},\epsilon_{18},\beta_6^{(3)}]&=(O_{18E},O_{18d},O_{18e})+\C[\gamma_{18}][\beta_6]+\C[\gamma_{18}][\delta_{18},\epsilon_{18},\beta_6^{(3)}]\\
&\subset \s{$I_{18}+\C[\beta_6,\gamma_{18}]+\C[\gamma_{18},\delta_{18}]+\C[\delta_{18},\epsilon_{18}]+\C[\epsilon_{18},\beta_6^{(3)}]$}.\end{align*}
Moreover, we see
\[\C[\alpha_{18},\beta_6,\gamma_{18},\delta_{18}]=(O_{18X}+9O_{18e},O_{18B})+\C[\beta_6][\alpha_{18}]+\C[\beta_6][\gamma_{18},\delta_{18}],\]
\begin{align*}\C[\alpha_{18},\beta_6,\gamma_{18},\delta_{18},\epsilon_{18},\beta_6^{(3)}]&=(O_{18C}-O_{18c},O_{18F})+\C[\beta_6,\gamma_{18},\delta_{18}][\alpha_{18}]\\
&\6+\C[\beta_6,\gamma_{18},\delta_{18}][\epsilon_{18},\beta_6^{(3)}],\end{align*}
\[\C[u_{18},\alpha_{18},\beta_6,\gamma_{18}]=(O_{18I},O_{18a})+\C[\alpha_{18}][u_{18}]+\C[\alpha_{18}][\beta_6,\gamma_{18}],\]
and
\begin{align*} R_{18}&=(O_{18A},O_{18b},O_{18c})+\C[\alpha_{18},\beta_6,\gamma_{18}][u_{18}]+\C[\alpha_{18},\beta_6,\gamma_{18}][\delta_{18},\epsilon_{18},\beta_6^{(3)}]\\
&=I_{18}+\C[u_{18},\alpha_{18}]+\C[\alpha_{18},\beta_6]+\C[\beta_6,\gamma_{18},\delta_{18},\epsilon_{18},\beta_6^{(3)}].\end{align*}

\subsection{The case N=10}

Put $R_{10}=\C[u_{10},\alpha_{10},\beta_{10},\epsilon_{10},\zeta_{10}]$ and
\[I_{10}=(O_{10a},O_{10b},O_{10c},O_{10d},O_{10e},O_{10f}).\]
We see
\[\C[u_{10},\alpha_{10},\beta_{10}]=(O_{10a})+\C[\alpha_{10}][u_{10}]+\C[\alpha_{10}][\beta_{10}],\]
\[\C[\beta_{10},\epsilon_{10},\zeta_{10}]=(O_{10f})\oplus\C[\beta_{10},\zeta_{10}]\oplus\C[\beta_{10},\zeta_{10}]\epsilon_{10},\]
\[\C[\alpha_{10},\beta_{10},\epsilon_{10},\zeta_{10}]=(O_{10e},O_{10c})+\C[\beta_{10}][\alpha_{10}]+\C[\beta_{10}][\epsilon_{10},\zeta_{10}],\]
and
\begin{align*} R_{10}&=(O_{10d},O_{10b})+\C[\alpha_{10},\beta_{10}][u_{10}]+\C[\alpha_{10},\beta_{10}][\epsilon_{10}\zeta_{10}]\\
&=I_{10}+\C[u_{10},\alpha_{10}]+\C[\alpha_{10},\beta_{10}]+\C[\beta_{10},\zeta_{10}]+\C[\beta_{10},\zeta_{10}]\epsilon_{10}.\end{align*}

\subsection{The case N=25}

Put $R_{25}=\C[u_{25},\alpha_{25},\iE_{i_5},\gamma_{25},\delta_{25},\iota_{25},\beta_5^{(5)}]$ and
\[I_{25}=(\s{$O_{25a},O_{25b},O_{25c},O_{25d},O_{25e},O_{25f},O_{25A},O_{25B},O_{25C},O_{25D},O_{25E},O_{25F},O_{25G},O_{25H},O_{25I}$}).\]
We see
\[\C[u_{25},\alpha_{25},\iE_{i_5}]=(O_{25a})+\C[\alpha_{25}][u_{25}]+\C[\alpha_{25}][\iE_{i_5}],\]
\[\C[\iE_{i_5},\gamma_{25},\delta_{25}]=(O_{25f})+\C[\gamma_{25}][\iE_{i_5}]+\C[\gamma_{25}][\delta_{25}],\]
\[\C[\alpha_{25},\iE_{i_5},\gamma_{25},\delta_{25}]=(O_{25c},O_{25e})+\C[\iE_{i_5}][\alpha_{25}]+\C[\iE_{i_5}][\gamma_{25},\delta_{25}],\]
and
\begin{align*}\C[u_{25},\alpha_{25},\iE_{i_5},\gamma_{25},\delta_{25}]&=(O_{25b},O_{25}d)+\C[\alpha_{25},\iE_{i_5}][u_{25}]+\C[\alpha_{25},\iE_{i_5}][\gamma_{25},\delta_{25}]\\
&\subset \s{$I_{25}+\C[u_{25},\alpha_{25}]+\C[\alpha_{25},\iE_{i_5}]+\C[\iE_{i_5},\gamma_{25}]+\C[\gamma_{25},\delta_{25}]$}.\end{align*}
Moreover, we see
\[\C[\delta_{25},\iota_{25},\beta_5^{(5)}]=(O_{25I})\oplus\C[\delta_{25},\beta_5^{(5)}]\oplus\C[\delta_{25},\beta_5^{(5)}]\iota_{25},\]
\[\C[\gamma_{25},\delta_{25},\iota_{25},\beta_5^{(5)}]=(O_{25D},O_{25H})+\C[\delta_{25}][\gamma_{25}]+\C[\delta_{25}][\iota_{25},\beta_5^{(5)}],\]
\[\C[\iE_{i_5},\gamma_{25},\delta_{25},\iota_{25},\beta_5^{(5)}]=(O_{25C},O_{25G})+\C[\gamma_{25},\delta_{25}][\iE_{i_5}]+\C[\gamma_{25},\delta_{25}][\iota_{25},\beta_5^{(5)}],\]
\begin{align*}\C[\alpha_{25},\iE_{i_5},\gamma_{25},\delta_{25},\iota_{25},\beta_5^{(5)}]&=(O_{25B},O_{25F})+\C[\iE_{i_5},\gamma_{25},\delta_{25}][\alpha_{25}]\\
&\6+\C[\iE_{i_5},\gamma_{25},\delta_{25}][\iota_{25},\beta_5^{(5)}],\end{align*}
and
\begin{align*} R_{25}&=(O_{25A},O_{25E})+\C[\alpha_{25},\iE_{i_5},\gamma_{25},\delta_{25}][u_{25}]+\C[\alpha_{25},\iE_{i_5},\gamma_{25},\delta_{25}][\iota_{25},\beta_5^{(5)}]\\
&=I_{25}+\C[u_{25},\alpha_{25},\iE_{i_5},\gamma_{25},\delta_{25}]+\C[\delta_{25},\beta_5^{(5)}]+\C[\delta_{25},\beta_5^{(5)}]\iota_{25}.\end{align*}

\subsection{The case N=7}

Put $R_7=\C[\iC_7,\alpha_7,\beta_7,\gamma_7,\delta_7]$ and
\[I_7=(O_{7a},O_{7b},O_{7c},O_{7d},O_{7e},O_{7f}).\]
We see
\[\C[\iC_7,\alpha_7,\beta_7]=(O_{7d})\oplus\C[\iC_7,\beta_7]\oplus\C[\iC_7,\beta_7]\alpha_7,\]
\[\C[\iC_7,\beta_7,\delta_7]=(O_{7a})+\C[\beta_7][\iC_7]+\C[\beta_7][\delta_7],\]
\[\C[\iC_7,\beta_7,\gamma_7,\delta_7]=(O_{7f})\oplus\C[\iC_7,\beta_7,\delta_7]\oplus\C[\iC_7,\beta_7,\delta_7]\gamma_7,\]
\[\C[\iC_7,\beta_7]\iC_7\gamma_7\subset(O_{7b})+\C[\iC_7,\alpha_7,\beta_7],\]
and
\begin{align*} R_7&=(O_{7e},O_{7c})+\C[\iC_7,\beta_7][\alpha_7]+\C[\iC_7,\beta_7][\gamma_7,\delta_7]\\
&=I_7+\C[\iC_7,\beta_7]+\C[\iC_7,\beta_7]\alpha_7+\C[\beta_7,\delta_7]+(\C[\iC_7,\beta_7]+\C[\beta_7,\delta_7])\gamma_7\\
&=I_7+\C[\iC_7,\beta_7]+\C[\iC_7,\beta_7]\alpha_7+\C[\beta_7,\delta_7]+\C[\beta_7,\delta_7]\gamma_7.\end{align*}

\section{Main results}

For reader's convenience, we review basic facts on graded rings. We say a ring $R$ is graded if $R$ is decomposed into a direct sum of additive groups
\[R=\VT{k=0}^\infty R_k\]
such that $R_kR_l\subset R_{k+l}$ for all $k,l\ge0$. In this paper, we only deal with the case $R_0=\C$. For example, $\C$ is graded as $\C_k=\{0\}$ for $k>0$, and so is $\cM(N)$ as $\cM(N)_k=\cM_k(N)$. Moreover, for a graded ring $R$ and $n_1,\cdots,n_r>0$, we define $S=R[X_1,\cdots,X_m]^{[n_1,\cdots,n_m]}$ to be a ring $R[X_1,\cdots,X_m]$ which is graded as $X_i\in S_{n_i}$. For given graded rings $R$ and $S$, a ring homomorphism $f:R\to S$ is said to be graded if $f(R_k)\subset S_k$ for $k\ge0$. In the sequel, every homomorphism is meant to be graded.

For a graded ring $R$ with $\dim R_k<\infty$ for all $k$, let
\[H(R)=\VS{k=0}^\infty(\dim R_k)t^k\in\Z[[t]]\]
be the Hilbert function of $R$. We see $H(\C)=1$ and $H(R[X]^{[n]})=H(R)(1-t^n)^{-1}$, in particular, $H(\C[X]^{[n]})=(1-t^n)^{-1}$ and
\[H(\C[X,Y]^{[2,2]})=\frac1{(1-t^2)^2}=H(\cM(4)).\]
\[H(\C[X,Y]^{[2,4]})=\frac1{(1-t^2)(1-t^4)}=H(\cM(2)),\]
\[H(\C[X,Y]^{[2,6]})=\frac1{(1-t^2)(1-t^6)}=\VS{k:{\rm even}}([\frac k6]+1)t^k,\]
\begin{align*} H(\C[X,Y]^{[4,6]})t^4&=\frac{t^4}{(1-t^4)(1-t^6)}\\
&=\frac1{(1-t^2)(1-t^4)}-\frac1{(1-t^2)(1-t^6)}\\
&=\VS{k:{\rm even}}\big([\frac k4]-[\frac k6]\big)t^k\\
&=\VS{k:{\rm even}}\big([\frac{k-4}{12}]+1-\delta_{12\Z+6}(k)\big)t^k\\
&=H(\cM(1))t^4.\end{align*}

Then the first main theorem is stated as follows:
\begin{thm} We have
\[\cM(1)\simeq\C[\iE_4,\iE_6]^{[4,6]},\]
\[\cM(2)\simeq\C[\iC_2,\alpha_2]^{[2,4]}=\C[\iC_2,\iE_4]^{[2,4]},\]
\[\cM(4)\simeq\C[\iC_2,\alpha_4]^{[2,2]}=\C[\iC_2,\iC_4]^{[2,2]}.\]
\end{thm}
\begin{proof}
In \S3.1 we have shown that the natural homomorphism
\[\C[\iE_4,\iE_6]^{[4,6]}\to\cM(1)\]
is surjective. By comparing dimensions on both sides, we easily see that the homomorphism is bijective. Thus we obtain
\[\C[\iE_4,\iE_6]^{[4,6]}\simeq\cM(1).\]
For $N=2$ and 4, the assertions can be shown in a similar way.
\end{proof}

Suppose that $R$ is a graded ring. If $I$ is an ideal of $R$ and $R=I\oplus S$ as $\C$-vector spaces, then for each $k$, we see $R_k=(I\cap R_k)\oplus(S\cap R_k)$. Moreover, if $I$ is homogeneous, then $R/I$ is naturally graded and we see
\[\dim(R/I)_k=\dim(R_k/(I\cap R_k))=\dim(S\cap R_k).\]
In particular, if $O\in R_{2k}+R_kY+Y^2$, then we have seen $R[Y]=(O)\oplus R\oplus RY$ in the previous section, and thus
\[H(R[Y]^{[k]}/(O))=H(R)(1+t^k).\]
We have the following results:
\begin{thm} We have
\[\cM(3)\simeq\C[\iC_3,\alpha_3,\beta_3]^{[2,4,6]}/(O_3),\]
\[\cM(5)\simeq\C[\iC_5,\alpha_5,\beta_5]^{[2,4,4]}/(O_5),\]
\[\cM(6)\simeq\C[\iC_3^{(2)},\alpha_6,\beta_6]^{[2,2,2]}/(O_6),\]
\[\cM(8)\simeq\C[\iC_4^{(2)},\alpha_4,\alpha_4^{(2)}]^{[2,2,2]}/(O_8),\]
\[\cM(9)\simeq\C[\iC_3,\alpha_9,\beta_9]^{[2,2,2]}/(O_9).\]
\end{thm}
\begin{proof}
In \S3.4 we have obtained the natural surjective homomorphism
\[\C[\iC_3,\alpha_3,\beta_3]^{[2,4,6]}\to\cM(3),\]
and we showed in \S5.6 that it induces the homomorphism
\[\C[\iC_3,\alpha_3,\beta_3]^{[2,4,6]}/(O_3)\twoheadrightarrow\cM(3).\]
We have
\begin{align*} H(\C[\iC_3,\alpha_3,\beta_3]^{[2,4,6]}/(O_3))&=\frac{1+t^4}{(1-t^2)(1-t^6)}\\
&=\VS{k:{\rm even}}\big(([\T\frac k6]+1)+([\frac{k-4}6]+1)\big)t^k\\
&=H(\cM(3)),\end{align*}
thus
\[\C[\iC_3,\alpha_3,\beta_3]^{[2,4,6]}/(O_3)\simeq\cM(3).\]
Similarly, since
\begin{align*} H(\C[\iC_5,\alpha_5,\beta_5]^{[2,4,4]}/(O_5))&=\frac{1+t^4}{(1-t^2)(1-t^4)}=\frac2{(1-t^2)(1-t^4)}-\frac1{1-t^2}\\
&=H(\cM(5)),\end{align*}
\[H(\C[\iC_3^{(2)},\alpha_6,\beta_6]^{[2,2,2]}/(O_6))=\frac{1+t^2}{(1-t^2)^2}=H(\cM(6)),\]
we may get the assertions.
\end{proof}
\smallskip
\begin{thm} We have
\[\cM(7)\simeq R_7^{[2,4,4,6,6]}/I_7,\]
\[\cM(10)\simeq R_{10}^{[2,2,2,4,4]}/I_{10},\]
\[\cM(12)\simeq R_{12}^{[2,2,2,2,2]}/I_{12},\]
\[\cM(16)\simeq R_{16}^{[2,2,2,2,2]}/I_{16}.\]
\end{thm}
\smallskip
\begin{proof}
Since
\[R_7=I_7\oplus\C[\iC_7,\beta_7]\oplus\C[\iC_7,\beta_7]\alpha_7\oplus\C[\beta_7,\delta_7]\delta_7\oplus\C[\beta_7,\delta_7]\gamma_7,\]
\[R_{10}=I_{10}\oplus\C[u_{10},\alpha_{10}]\oplus\C[\alpha_{10},\beta_{10}]\beta_{10}\oplus\C[\beta_{10},\zeta_{10}]\zeta_{10}\oplus\C[\beta_{10},\zeta_{10}]\epsilon_{10},\]
\[R_{12}=I_{12}\oplus\C[\iC_3^{(2)},\alpha_6]\oplus\C[\alpha_6,\beta_6]\beta_6\oplus\C[\beta_6,\gamma_{12}]\gamma_{12}\oplus\C[\gamma_{12},\beta_6^{(2)}]\beta_6^{(2)},\]
we see
\begin{align*} H(R_7^{[2,4,4,6,6]}/I_7)&=\frac{1+t^4}{(1-t^2)(1-t^4)}+\frac{2t^6}{(1-t^4)(1-t^6)}\\
&=\frac{1+2t^4+t^6}{(1-t^2)(1-t^6)}\\
&=2\frac{1+t^4}{(1-t^2)(1-t^6)}-\frac1{1-t^2}\\
&=H(\cM(7)),\end{align*}
\begin{align*} H(R_{10}^{[2,2,2,4,4]}/I_{10})&=\frac{1+t^2}{(1-t^2)^2}+\frac{2t^4}{(1-t^2)(1-t^4)}\\
&=\frac{2t^2}{(1-t^2)^2}+\frac{1+t^4}{(1-t^2)(1-t^4)}\\
&=H(\cM(10)),\end{align*}
\[H(R_{12}^{[2,2,2,2,2]}/I_{12})=\frac{1+3t^2}{(1-t^2)^2}=H(\cM(12)),\]
and we get the assertions.
\end{proof}

\begin{thm} We have
\[\cM(18)\simeq R_{18}^{[2,2,2,2,2,2,2]}/I_{18},\]
\[\cM(25)\simeq R_{25}^{[2,2,2,2,2,4,4]}/I_{25}.\]
\end{thm}
\begin{proof}
Since
\begin{align*}R_{18}&=I_{18}\oplus\C[u_{18},\alpha_{18}]\oplus\C[\alpha_{18},\beta_6]\alpha_{18}\\
&\6\oplus\C[\beta_6,\gamma_{18}]\beta_6\oplus\C[\gamma_{18},\delta_{18}]\gamma_{18}\oplus\C[\delta_{18},\epsilon_{18}]\delta_{18}\oplus\C[\epsilon_{18},\beta_6^{(3)}]\epsilon_{18},\end{align*}
\begin{align*}R_{25}&=I_{25}\oplus\C[u_{25},\alpha_{25}]\oplus\C[\alpha_{25},\iE_{i_5}]\iE_{i_5}\oplus\C[\iE_{i_5},\gamma_{25}]\gamma_{25}\oplus\C[\gamma_{25},\delta_{25}]\delta_{25}\\
&\6\oplus\C[\delta_{25},\beta_5^{(5)}]\beta_5^{(5)}\oplus\C[\delta_{25},\beta_5^{(5)}]\iota_{25},\end{align*}
we see
\[H(R_{18}^{[2,2,2,2,2,2,2]}/I_{18})=\frac{1+5t^2}{(1-t^2)^2}=H(\cM(18)),\]
\begin{align*} H(R_{25}^{[2,2,2,2,2,4,4]}/I_{25})&=\frac{1+3t^2}{(1-t^2)^2}+\frac{2t^4}{(1-t^2)(1-t^4)}\\
&=H(\cM(25)),\end{align*}
and we get the assertions.
\end{proof}

\section{Integrality of the basis}

We easily see that the basis $\{b_j\}$ taken in \S3 is rational, namely, for each $j$
\[b_j\in q^{j-1}+\Q[[q]]q^j.\]
In this section, we would prove that those are integral, namely, for each $j$
\[b_j\in q^{j-1}+\Z[[q]]q^j.\]
Indeed, when $N=1$, the assertion follows from $\iE_4,\iE_6,\Delta\in\Z[[q]]$. Once we may take such an integral basis $\{b_j\}$, we easily obtain that
\[\cM_k(N)\cap(\Z\oplus\Z q\oplus\cdots\oplus\Z q^{d-1}\oplus\C[[q]]q^d)\subset\Z[[q]]\]
for each $k>0$ and $N\in\{1,2,3,4,5,6,7,8,9,10,12,16,18,25\}$, where $d=\dim\cM_k(N)$.

\subsection{The case N=4,6,8,9,12,16,18}

We see
\[\tau_{N^2}(n)-\tau_N(n)=\VS{d|n,N|d,N^2\nmid d}d=N\tau_N(\frac nN),\]
where we make the convention that $\tau_N(\frac nN)=0$ if $N\nmid n$. In particular, for each prime number $p$, we see
\[\T\frac1p((p+1)\tau_p-\tau_{p^2})=\tau_p-\frac1p(\tau_{p^2}-\tau_p)=1_p\cdot\sigma_1,\]
where $1_p=\delta_{\N\diagdown p\N}$, since $\tau_p(pn)=\tau_p(n)$. Therefore we get
\[\alpha_4=\VS{n}\frac12(3\tau_2-\tau_4)(n)q^n=\VS{n\equiv1\bmod2}\sigma_1(n)q^n\in\Z[[q]],\]
\begin{align*}\beta_9&=\T\frac16\bigg(\VS{n}\frac13(4\tau_3-\tau_9)(n)q^n-\iE_{\rho_3}\bigg)\\
&=\VS{n\equiv2\bmod3}\frac13\sigma_1(n)q^n\\
&=\VS{n\equiv2\bmod3}\VS{d|n,d\equiv1\bmod3}\frac13(d+\frac nd)q^n\in\Z[[q]],\end{align*}
\begin{align*}\gamma_{16}&=\VS{n\equiv3\bmod4}\frac14\sigma_1(n)q^n\\
&=\VS{n\equiv3\bmod4}\VS{d|n,d\equiv1\bmod4}\frac14(d+\frac nd)q^n\in\Z[[q]].\end{align*}
We note that the integrality of $\gamma_{16}$ can be proved also from
\[\gamma_{16}=(\alpha_4^{(2)}+4\alpha_4^{(4)})\alpha_4/\iC_4^{(2)}\in\Z[[q]].\]
Moreover we easily reconfirm that our tuples taken in \S3.1 are integral.

\subsection{The case N=2,5,10,25}

We see that our tuples taken in \S3.2 are integral, for example,
\begin{align*}\beta_5&=\T\frac13\Big(-\VS{n}\tau_5(n)q^n+\VS{n}\sigma_3(n)(q^n-q^{5n})+20\VS{n}\sigma_3(n)q^{5n}\Big)\\
&\6\6-\Big(\VS{n}\tau_5(n)q^n\Big)^2\\
&=\VS{n}\Big(\VS{d|n,5\nmid d}\T\frac13(d^3-d)+48\sigma_3(\frac n5)\Big)q^n-\Big(\VS{n}\tau_5(n)q^n\Big)^2\\
&\in\Z[[q]],\end{align*}
\begin{align*}\delta_{25}&=\T\frac1{100}\VS{n}(6\tau_5(n)-\tau_{25}(n))q^n+\frac1{20}\iE_{\rho_5}-\frac1{10}\iE_{r_5}\\
&=\T\frac1{20}\VS{5\nmid n}(\sigma_1(n)+\rho_5(n)\sigma_1(n))q^n-\frac1{10}\iE_{r_5}\\
&=\T\frac1{10}\VS{n\equiv1\bmod5}(\sigma_1(n)-\sigma_{\rho_5}(n))q^n+\frac1{10}\VS{n\equiv4\bmod5}(\sigma_1(n)-\sigma_{\rho_5}(n))q^n\\
&=\VS{n\equiv1\bmod5}\VS{d|n,d\equiv2\bmod5}\T\frac15(d+\frac nd)q^n+\VS{n\equiv4\bmod5}\VS{d|n,d\equiv1\bmod5}\frac15(d+\frac nd)q^n\\
&\in\Z[[q]]\end{align*}

\subsection{The case N=3,7}

We see that our tuples taken in \S3.3 are integral, for example,
\begin{align*}\beta_7&=\T\frac14\Big(-\VS{n}\tau_7(n)q^n+\VS{n}\sigma_3(n)(q^n-q^{7n})+30\VS{n}\sigma_3(n)q^{7n}\Big)\\
&\6\6-\T\frac12\Big(\VS{n}\tau_7(n)q^n\Big)^2\\
&=\VS{n}\Big(\VS{d|n,7\nmid d}\T\frac14(d^3-d)+93\sigma_3(\frac n7)\Big)q^n-\frac12\Big(\VS{n}\tau_7(n)q^n\Big)^2\\
&\in\Z[[q]]+\VS{n}\VS{d|n,7\nmid d}\T\frac14(d^3-d)q^n-\VS{n}\frac12\tau_7(n)^2q^{2n}\\
&=\Z[[q]]+\VS{n}\Big(\VS{d|n,2|d,7\nmid d}\T\frac14(-d)-\VS{d|\frac n2,7\nmid d}\frac 12d^2\Big)q^n\\
&=\Z[[q]]+\VS{n}\Big(\VS{d|\frac n2,7\nmid d}\frac 12(-d-d^2)\Big)q^n\\
&=\Z[[q]],\end{align*}
\[\delta_7=\iC_7^{-1}\beta_7^2\in\Z[[q]].\]

\end{document}